\documentclass[11pt]{amsart}
\usepackage{amsmath,amssymb,amsthm,amscd,xy}
\usepackage{enumerate}

\title{Equivalences induced by infinitely generated silting modules}

\author{Simion Breaz}
\address{Babe\c s--Bolyai University, Faculty of Mathematics and Computer Science \\  1, Mihail Kog\u alniceanu, 400084 Cluj--Napoca, Romania}
\email{bodo@math.ubbcluj.ro}

\author{George Ciprian Modoi}
\address{Babe\c s--Bolyai University, Faculty of Mathematics and Computer Science \\  1, Mihail Kog\u alniceanu, 400084 Cluj--Napoca, Romania}
\email{cmodoi@math.ubbcluj.ro}

\thanks{}

\subjclass[2010]{16E30, 18E30, 16D90}
\keywords{silting module, silting complex, endomorphism ring, endomorphism dg-algebra, dg-module, derived functors}

\renewcommand{\iff}{if and only if }

\newcommand{\la}{\longrightarrow}

\newcommand{\Z}{\mathbb{Z}}

\DeclareMathOperator{\Hom}{Hom} 
\DeclareMathOperator{\RHom}{RHom}
\newcommand{\dotimes}{\otimes^{\mathrm{L}}}
\DeclareMathOperator{\End}{End}
\DeclareMathOperator{\DgEnd}{DgEnd}
\DeclareMathOperator{\Ext}{Ext} 
\DeclareMathOperator{\Tor}{Tor}
\DeclareMathOperator{\Ker}{Ker}
\DeclareMathOperator{\Img}{Im}

\DeclareMathOperator{\Coker}{Coker}

\DeclareMathOperator{\Zy}{\mathrm{Z}}
\DeclareMathOperator{\Ho}{\mathrm{H}}
\DeclareMathOperator{\Bd}{\mathrm{B}}

\DeclareMathOperator{\KT}{\mathbf{Z}}

\newcommand{\A}{\mathcal{A}}
\newcommand{\B}{\mathcal{B}}
\newcommand{\K}{\mathcal{K}}

\newcommand{\C}{\mathcal{C}}
\newcommand{\D}{\mathcal{D}}

\newcommand{\F}{\mathcal{F}}
\newcommand{\CH}{\mathcal{H}}

\newcommand{\T}{\mathcal{T}}
\newcommand{\CT}{\mathcal{T}}

\newcommand{\U}{\mathcal{U}}

\newcommand{\CU}{\mathcal{U}}
\newcommand{\CV}{\mathcal{V}}

\newcommand{\BP}{\mathbb{P}}

\newcommand{\BQ}{\mathbb{Q}}
\newcommand{\BE}{\mathbb{E}}

\newcommand{\bD}{\mathbf{D}}

\newcommand{\Modr}{\mathrm{Mod}\text{-}}
\newcommand{\Def}{\mathrm{Def}}
\newcommand{\Gen}{\mathrm{Gen}}

\newcommand{\opp}{^\textit{op}}
\newcommand{\Mod}[1]{\hbox{\rm Mod}({#1})}
\newcommand{\DgMod}[1]{\hbox{\rm DgMod}({#1})}

\newcommand{\add}[1]{\mathrm{add}({#1})}
\newcommand{\Add}[1]{\mathrm{Add}({#1})}

\newcommand{\Der}[1]{\mathbf{D}({#1})}
\newcommand{\Htp}[1]{\mathbf{K}({#1})}

\theoremstyle{plain}
\newtheorem{thm}{Theorem}[subsection]
\newtheorem{lem}[thm]{Lemma}
\newtheorem{prop}[thm]{Proposition}
\newtheorem{cor}[thm]{Corollary}

\theoremstyle{definition}

\theoremstyle{remark}
\newtheorem{rem}[thm]{Remark}

\newtheorem{expl}[thm]{Example}

\begin{document}

\begin{abstract}
We study equivalences induced by a complex $\BP$, consisting of projectives and concentrated in degrees $-1$ and $0$, which 
is silting in the derived category $\bD(R)$ of a ring $R$. 
\end{abstract}

\maketitle

\section{Introduction}

A \textsl{torsion theory} in an abelian category $\A$ 
(e.g. $\A=\Mod R$ is the category of right $R$-modules) is a pair $\tau=(\T,\F)$, such that the classes $\T$ and $\F$ are  $\Hom_\A(\T,\F)=0$, and for every $X\in\A$ there is a short exact sequence 
$0\to T\to X\to F\to 0$, such that $T\in\T$ and $F\in\F$. Then $\T$ and $\F$ are called the \textsl{torsion class}, respectively 
the \textsl{torsion free class} of $\tau$. 

In the context of a triangulated category 
$\D$ endowed with the shift functor $-[1]:\D\to\D$ (e.g. $\D=\Der R$ the derived category of the category of $R$-modules), a \textsl{t-structure} is a pair $(\A,\B)$ of full subcategories if $\D$ such that 
\begin{enumerate}
 \item $\Hom_\D(\A,\B[-1])=0$.
 \item $\A\subseteq\A[-1]$ (or equivalently $\B[-1]\subseteq\B$).
 \item For every $X\in\D$ there is a triangle $X'\to X\to X''\stackrel{+}\to$, where $X'\in\A$ and $X''\in\B[-1]$.
\end{enumerate}
%
The \textsl{heart} of a t-structure $(\A,\B)$ is defined to be the subcategory 
$\CH=\A\cap\B$. We recall that the heart $\CH$ is an abelian category. Note that the definition of a t-structure implies 
immediately that the inclusion functors $\A\to\D$ and $\B\to\D$ have a right, respectively a left adjoint. 
For more informations about torsion pairs and t-structures one can consult  \cite[Chapter I, Section 2]{HRS}.

One of the central results in Tilting Theory is the Tilting Theorem, \cite[Theorem 3.5.1]{Colby_Fuller:2004}, which states that if $(\T,\F)$ the torsion theory generated by a
finitely presented (i.e. classical) tilting right $R$-module $T$ then there exists a torsion theory $(\mathcal{X},\mathcal{Y})$ in the category of right $E$-modules 
($E$ is the endomorphism ring of $T$) and a pair of equivalences 
$$\Hom_R(T,-):\T\rightleftarrows \mathcal{Y}:-\otimes_{E}T\text{ and }\Ext^1_R(T,-):\T\rightleftarrows \mathcal{X}:\Tor^{E}_1(-,T).$$ Such a pair of equivalences is called a 
\textsl{counter-equivalence}. It was proved in \cite{CET} that the existence of a counter equivalence is strongly related to the existence of a classical tilting module which generates $\T$. 
In the case of infinitely generated tilting modules, versions of Tilting Theorem was formulated at the level of for derived category in \cite{Baz1} and \cite{BMT}.  The main idea is that every tilting module $S$ is equivalent to a good tilting module $T\in\Mod{R}$ which induces an equivalence between the derived category $\Der{R}$ and a subcategory 
of the derived category $\Der{\End(T)}$. This equivalence also induces a counter equivalence at the level of module categories, i.e. the functors $\Hom_R(T,-)_{|\T}$ and $\Ext^1_R(T,-)_{|\F}$ are fully faithful, 
 and the quasi-inverses of the functors are induced by $-\otimes_ET$ and $\Tor_1^E(T,-)$, \cite{Baz1}, \cite{fac1}, \cite{fac2}.

In order to be more precise, let us start with some settings and well--known definitions. In this paper all rings are unital, all categories and functors are additive, and all classes of objects are closed under isomorphisms. If $R$ is a ring then $\Mod{R}$
denotes the category of right $R$-modules, and $\Der R$ is the associated derived category of $\Mod R$.  If $\BP$ is a complex, then $\Ho^{n}(\BP)$ 
denotes the $n$-th cohomology group associated to $\BP$.
If $\C$ is a category and $X$ is an object in $\C$ then $\Add X$ (resp. $\add X$) denotes the class of all objects isomorphic to direct summands of 
(finite) direct sums of copies of $X$. If $F:\C\to \D$ is a functor then $\Ker F$ denotes the class of all objects $X$ from $C$ such that $F(X)=0$. 
If $T$ and $M$ are $R$-modules then $M$ is \textsl{$T$-generated} if there exists an epimorphism $T^{(I)}\to M$, and  $\Gen(T)$ denotes the class of all $T$-generated modules.

Now we return to the case of an $R$-module $T$, with $E=\End_R(T)$. If $T$ is tilting, then the torsion theory $(\T,\F)$ associated with $T$ has $\T=\Gen(T)$. 
By \cite[Chapter 1, Proposition 2.1]{HRS} it induces a t-structure in the derived category $\Der{R}$ of $R$, whose heart $\mathcal{H}$ is equivalent to the category of right $E$-modules. 
This equivalence is realized by the derived Hom functor $\RHom_R(T,-)$, 
and its quasi-inverse is computed by using the derived tensor product.
Conversely, it was proved that the heart of the $t$-structure associated to a torsion theory is equivalent to a module category if the torsion class is generated by a module $T$ 
which has a projective presentation
$P^{-1}\to P^{0}\to T\to 0$  such that the associated complex $$\BP=\dots\to 0\to P^{-1}\to P^{0}\to 0\to\cdots$$ has some special properties (it is compact and silting) 
in the derived category (\cite{CMT}, \cite{HKM}, \cite{MT}, \cite{PS1}). 
In particular, the support $\tau$-tilting modules introduced in \cite{AIR} admit such a projective presentation.

Silting modules are generalizations of tilting ones and they were introduced in \cite{AMV} as infinitely generated versions of support $\tau$-tilting modules. Further  they are characterized as the modules of the form $\Ho^0(\BP)$, where $\BP$ is a two term silting complex. We refer to \cite{IJY}, \cite{Ma-Sto:2015}, and \cite{Ps-Vi} 
for various correspondences realized by such complexes. The main aim of the present paper is to study some equivalences induced by silting modules, providing a Silting Theorem, that is a correspondent for the Tilting Theorem. This can be useful since for perfect or hereditary rings many torsion theories
are generated by silting modules, \cite{BrZ2}, but there are many of them which are not generated by tilting modules, \cite{An-ab}. It was proved in \cite[Theorem 3.8]{Jasso} that the Hom-covariant functor and the tensor functor induced by a support
$\tau$-tilting module define an equivalence as in the above described counter-equivalences. If $T$ is a silting module then it still induces a torsion pair $(\T,\F)$, where 
$\T=\Gen(T)$. If $U$ is the annihilator ideal for $\T$ 
then $T$ is an $R/U$-tilting module (possible infinitely generated), 
hence the Tilting Theorem proved in \cite{Baz1} can be applied to deduce that $\Hom_R(T,-):\Modr R\to \Modr E$ induces an equivalence between $\T$
and its image with $-\otimes_E T$ as a quasi-inverse. But a direct application of the Tilting Theorem does not give us information for the whole class $\F$.
For a support $\tau$-tilting module $T$, the case when the covariant functors $\Ext^1_R(T,-)$ and $\Tor_1^E(-,T)$ induce an equivalence is characterized in \cite{Tre}. 
If $R$ is hereditary, by \cite[Proposition 5.2]{An-M-V} it follows that the annihilator of a silting module is idempotent, and it is easy to see using \cite[Theorem 2.1]{Tre} 
that in the case of support $\tau$-tilting modules the covariant functor $\Ext^1_R(T,-)_{|\F}$ induces an equivalence with the quasi-inverse the functor 
$\Tor_1^E(-,T)$ iff the module $T$ is tilting (at the level of derived categories the same conclusion can be obtained by using \cite[Theorem A]{Ps-Vi}). 

The main aim of this paper is to study the equivalences induced by a silting module associated to a silting complex $\BP$, and to extend the results proved in \cite{Buan} and \cite{HKM} for the support $\tau$-tilting case. In contrast with the tilting case, when we consider a silting object $\BP\in\Der R$, 
the module $T=\Ho^0(\BP)$ does not carry all information we need since a silting complex is not quasi-isomorphic (i.e. it is not isomorphic in the derived category) to the corresponding silting module. Therefore we have to deal not only with the module $T$ but with the whole complex $\BP$.
In Section \ref{silting-basics} are gathered necessary results about dg-modules over dg-algebras. In Section \ref{silting-complex} 
we recall the definitions for silting modules and silting complexes, and some basic properties connected to  the torsion theory $(\T,\F)$ associated with such a module (complex). 
Next we construct a
\textsl{good silting module} $T$ which generates $\T$. Therefore, every silting complex will be equivalent (in the sense that they induce the same torsion theory) to a good one. If 
$\BP=P^{-1}\overset{\sigma}\to {P^0}$ is a silting complex, 
for which $\Ho^0(\BP)=T$, then we consider the right derived 
Hom functor and its left adjoint (namely the left derived tensor product) between the category $\Der R$ and the derived category of the dg-endomorphism algebra $\DgEnd_R(\BP)$ of $\BP$. In Section \ref{silting-thm} we state and prove the targeted Silting Theorem, first at the level of the derived categories, that is for a good silting complex of $R$-modules $\BP$ we construct an equivalence between $\Der R$ and a subcategory of $\Der{B}$ where $B$ is a smart truncation of  $\DgEnd_R(\BP)$. 
Then we specialize the 
above equivalence, in order to obtain the so called, a silting counter equivalence between the torsion theory induced by a silting module and some subcateogries of the torsion-free class, repectively the torsion class of the torsion pair $(\CU,\CV)$ in $\Mod\BE$ which is defined by $\CU=\Ker(-\otimes_\BE T)$.

\section{Preliminaries}\label{silting-basics}



We recall here some generalities about dg-algebras and the total derived functors between their derived categories. 
We will follow 
\cite{KDG}, \cite{KDU}, and \cite{keller-functori} in these considerations.

Let $k$ be a commutative ring. Recall that a \textsl{dg-algebra} is a $\Z$-graded $k$-algebra $B=\bigoplus_{i\in\Z}B^i$ endowed with a \textsl{differential}  
$d:B\to B$ such that $d^2=0$ which is homogeneous of degree $1$, that is   $d(B^i)\subseteq B^{i+1}$ for all $i\in\Z$, and satisfies the graded Leibniz rule:
\[d(ab)=d(a)b+(-1)^iad(b)\hbox{, for all }a\in B^i\hbox{ and }b\in B.\] 
A (by default, right) \textsl{dg-module} over $B$ is a $\Z$-graded module \[M=\bigoplus_{i\in\Z}M^i\] endowed with a $k$-linear square-zero differential $d:M\to M$, which is homogeneous 
of degree $1$ and satisfies the graded Leibnitz rule: 
\[d(xb)=d(x)b+(-1)^ixd(b)\hbox{, for all }x\in M^i\hbox{ and }b\in B.\] Left dg-$B$-modules are defined similarly. A \textsl{morphism of dg-$B$-modules} 
is a $B$-linear map $f:M\to N$ compatible with gradings and differentials. In this way we obtain the category $\Mod{B}$ of all dg-$B$-modules. 

If $B$ is a dg-algebra, then the \textsl{dual dg-algebra} $B\opp$ is defined as follows:
as graded $k$-modules $B\opp=B$, the multiplication is given by $ab=(-1)^{ij}ba$ for all $a\in B^i$ and all $b\in B^j$ and the differential $d:B\opp\to B\opp$ is the same 
as in the case of $B$. It is clear that a left dg-$B$-module $M$ is a 
right dg-$B\opp$-module with the ``opposite" multiplication $xa=(-1)^{ij}ax$, for all $a\in B^i$ and all $x\in M^j$, henceforth we denote by 
$\Mod{B\opp}$ the category of left dg-$B$-modules.

For a dg-module $M\in\Mod{B}$ and for all $n\in\Z$ we define the $n$-th \textsl{cocycles}, \textsl{boundaries}, respective \textsl{cohomology} $B^0$-modules by 
\begin{align*} &\Zy^n(M)=\Ker(M^n\stackrel{d}\to M^{n+1}), \  
\Bd^n(M)=\Img(M^{n-1}\stackrel{d}\to M^{n}), 
\textrm{ and }\\ & \Ho^n(M)=\Zy^n(M)/\Bd^n(M).\end{align*} 
Note that these formulas induce functors into the category of $B^0$-modules. 

A morphism of dg-modules is called \textsl{quasi-isomorphism} if it induces 
isomorphisms in all cohomologies. A dg-module $M\in\Mod{B}$ is \textsl{acyclic} if $\Ho^n(M)=0$ for all $n\in\Z$. 
A morphism of dg-$B$-modules $f:M\to N$ is called \textsl{null--homotopic} provided that there is a graded homomorphism $s:M\to N$ of degree $-1$ such that 
$f=sd+ds$. The \textsl{homotopy category} $\Htp{B}$ has the same objects as $\Mod{B}$ and the morphisms are equivalence classes of morphism of dg-modules, up to homotopy. 
It is well--known that the homotopy category is triangulated. Moreover a null--homotopic morphism is acyclic, therefore the functors $\Ho^n$ factor through $\Htp{B}$ 
for all $n\in\Z$.

The \textsl{derived category} $\Der{B}$ is obtained from $\Htp{B}$ by formally inverting all quasi-isomorphisms. An object $U\in\Der{B}$ is called \textsl{cofibrant} 
if for every acyclic dg-$B$-module $N$ we have  
$\Hom_{\Htp{B}}(U,N)=0$. This is equivalent to $$\Hom_{\Der{B}}(U,M)=\Hom_{\Htp{B}}(U,M)$$ for all dg-$B$-modules $M$. Dually we define fibrant objects. 

For two dg-modules $M,N\in\Mod{B}$ we consider the so called \textsl{dg-Hom} complex 
\[\Hom^\bullet_B(M,N)=\bigoplus_{n\in\Z}\Hom_B^n(M,N) \]
with $\Hom_B^n(M,N)=\prod_{i\in\Z}\Hom_B(M^i,N^{n+i})$, whose differentials are given by 
$$d(f)(u)=d_Nf(u)-(-1^n)fd_M(u) \textrm{ for all } f\in\Hom_B^n(M,N).$$ In this way we obtain a new category, $\DgMod{B}$ whose objects are the same as the objects of 
$\Mod{B}$, that is dg-modules, but whose morphisms are dg-Hom complexes. Note that the morphisms in $\Mod{B}$ and $\Htp{B}$ between the dg-modules $M$ and $N$, are exactly 
$\Zy^0\Hom^\bullet_B(M,N)$, respectively $\Ho^0\Hom^\bullet_B(M,N)$.

Let now $A$ and $B$ be two dg-algebras and let $U$ be a dg-$B$-$A$-bimodule (that is $U$ is a dg-$B\opp\otimes_kA$-module).  
In this situation, for every $X\in\Mod{A}$ the dg-Hom complex  $\Hom^\bullet_A(U,X)$ becomes a dg-$B$-module, so we get a functor 
(the definition on morphisms is obvious)
\[\Hom^\bullet_A(U,-):\Mod{A}\to\Mod{B}.\]
It induces the \textsl{right derived Hom functor}  
\[\RHom_A(U,-):\Der{A}\to\Der{B},\]
where $\RHom_A(U,X)=\Hom^\bullet_A(U',X)\cong\Hom_A^\bullet(U,X')$ where $U'$ is a \textsl{cofibrant replacement} of $U$ (that is, a cofibrant dg-$A$-module $U'$ 
together with a  quasi-isomorphism $U'\to U$) and $X'$ is a fibrant replacement of $X$ (which is defined by duality), see \cite[Theorem 12.1.1]{Yek}. It was proved in \cite[Theorem 3.1]{KDG} that
(co)fibrant replacements always exist in $\Htp{A}$.  

 Let $M\in\Mod{B}$. There exists a natural grading on the usual tensor product $M\otimes_BU$, which can be described as:
\[M\otimes^\bullet_BU=\bigoplus_{n\in\Z}M\otimes^n_BU,\]
where $M\otimes^n_BU$ is the quotient of $\bigoplus_{i\in\Z}M^i\otimes_{B^0}U^{n-i}$ 
by the submodule generated by $m\otimes bu-mb\otimes u$ where $m\in M^i$, $u\in U^{j}$ and $b\in B^{n-i-j}$, for all $i,j\in\Z$. Together with the differential
\[d(m\otimes u)=d(m)\otimes u+(-1)^im\otimes d(u)\hbox{, for all }m\in M^i, u\in U,\] we obtain a 
a functor   
$-\otimes^\bullet_BU:\Mod{B}\to\Mod{A},$ and further a triangle functor 
$-\otimes^\bullet_BU:\Htp{B}\to\Htp{A}.$ The \textsl{left derived tensor product} \[-\dotimes_BU:\Der{B}\to\Der{A}\] is defined 
by $Y\dotimes_BU=Y'\otimes^\bullet_BU\cong Y\otimes^\bullet_BU',$ where $Y'$ and $U'$ are 
cofibrant replacements for $Y$ and $U$ in $\Htp{B}$ and $\Htp{B\opp}$ respectively.

A dg-algebra $B=\bigoplus_{i\in\Z}B^i$ is called \textsl{(homologically) non-positive} if $B^i=0$ (respectively $\Ho^i(B)=0$) for $i>0$.

\section{Two term silting complexes}\label{silting-complex} 

\subsection{Silting modules and silting complexes} Let $R$ be a unital ring. If $P^{-1}\overset{\sigma}\to {P^0}$ is a morphism between projective right $R$-modules then \textsl{the defect of $\sigma$} 
is defined as the functor $$\Def_\sigma(-)=\Coker(\Hom_R(\sigma,-)):\Mod{R}\to Ab.$$ We will denote by $\D_\sigma$ the kernel (on objects) of $\Def_\sigma$, 
i.e. the class of all modules $L\in \Mod{R}$ such that every morphism $\alpha:P^{-1}\to L$ can be extended to a morphism $P^0\to L$. 

We recall from \cite{AMV} that a right $R$-module $T$ is \textsl{silting} with respect to a projective resolution $P^{-1}\overset{\sigma}\to {P^0}\to T\to 0$ if 
\begin{itemize}
\item[(s)] $\Gen(T)=\D_\sigma$. 
\end{itemize}
It is easy to see that $\Gen(T)$ is closed under direct sums and epimorphic images. Using \cite[Proposition 4]{BrZ}, it follows that $\Ker(\Def_\sigma)$ is closed under extensions. 
Therefore, if $T$ is silting with respect to $\sigma$ then the class $\T=\Gen(T)=\D_\sigma$ is a torsion class. We will denote by $\tau=(\T,\F)$ the induced torsion theory in $\Mod{R}$. 

In this case we associate to $\sigma$ the complex 
$$\BP= \dots\to 0\to P^{-1}\overset{\sigma}\to {P^0}\to 0\to \cdots$$ 
of projective modules, and we note that $T$ is silting with respect to $\sigma$ if and only if $\BP$ is a
\textsl{silting complex} of projective modules (cf. \cite[Theorem 4.9]{AMV}), i.e. 
\begin{itemize}
\item[(S1)] $\BP^{(I)}\in \BP^{\perp_{>0}}$ for all sets $I$, and 
\item[(S2)] the homotopy category $\mathbf{K}^b (\mathrm{Proj(R)})$ is the smallest triangulated subcategory of $\bD(R)$ containing $\Add\BP$,
\end{itemize}
where $$\BP^{\perp_{>0}}=\{Y\in \Der{R}\mid \Hom_{\Der{R}}(\BP,Y[n])=0 \text{ for all positive integers }n\}.$$
If $\BP$ satisfies only the condition (S1) then it is called \textsl{presilting}.

\begin{rem} In literature a complex as before is also called 2-term silting complex, in order to emphasize that it contains only two non-zero entries. 
There are also defined $n$-term silting complexes, which are complexes with $n$ non-zero entries satisfying (S1) and (S2). We refer to \cite{An18} for a recent survey on this subject. However \textit{in what follows 
we entirely stick to the case of a 2-term silting complex, hence we drop the expression ``2-term'' from our considerations.}
\end{rem}

The following lemma is straightforward. It records connections between the functors induced by $T$ and $\BP$.

\begin{lem}\label{basic-hom-def}
Let $\BP\in\Der R$ be a complex induced by a morphism $\sigma:P^{-1}\to P^0$ between projective modules, and denote $T=\Ho^0(\BP)$. Then for every $M\in\Modr R$ there are canonical isomorphisms:
\begin{enumerate}
 \item   $\Hom_{\Der R}(\BP, M)\cong \Hom_R(T,M)$; 
\item  $\Hom_{\Der R}(\BP, M[1])\cong \Def_\sigma(M)$;
\item  $T\cong \Hom_{\Der{R}}(R,\BP)$.
 \end{enumerate}
\end{lem}

From \cite[Theorem 4.6]{AMV} we extract the following useful result:

\begin{lem}\label{addP}
 If $\BP$ is a silting complex then 
 $$\Add{\BP}=\{X\in \BP^{\perp_{>0}}\mid \Hom_{\Der{R}}(X,Y[1])=0 \text{ for all } Y\in\BP^{\perp_{> 0}}\}.$$ 
\end{lem}

The following result, which is a generalization of \cite[Corollary 3.3]{Buan}, can be extracted from \cite{Wei-semi}.
We include a proof for reader's convenience.

\begin{prop}\label{triangle}
Let $\BP\in\Der R$ be a complex induced by a morphism $\sigma:P^{-1}\to P^0$ between projective modules. The following are equivalent:
\begin{enumerate}[{\rm (1)}]
 \item $\Ho^0(\BP)$ is a silting module with respect to $\BP$;
 \item there exists a triangle $R\to \BP'\to \BP''\to R[1]$ in $\Der R$ such that $\BP'$ and $\BP''$ are in $\Add \BP$.
\end{enumerate}
\end{prop}

\begin{proof}
(1)$\Rightarrow$(2) Let $I=\Hom_{\Der{R}}(\BP,R[1])$, and we consider a triangle 
$$ R\to \BQ\overset{\beta}\to \BP^{(I)}\to R[1]$$ induced by the canonical $\Add \BP$-precovering $\BP^{(I)}\overset{\beta}\to R[1]$ (this means that $\Hom_{\Der{R}}(\BP,\beta)$ is an epimorphism). 
Applying the functor $\Hom_{\Der{R}}(\BP,-)$ to the above triangle, we obtain the exact sequence of $k$-modules
\begin{align*} &\Hom_{\Der{R}}(\BP,R[i])\to 
\Hom_{\Der{R}}(\BP,\BQ[i])\to \Hom_{\Der{R}}(\BP,\BP^{(I)}[i])\to \\
\to & \Hom_{\Der{R}}(\BP,R[i+1])\to 
\Hom_{\Der{R}}(\BP,\BQ[i+1])\to \Hom_{\Der{R}}(\BP,\BP^{(I)}[i+1])\end{align*}
for all $i\geq 0$. 
Since $\Hom_{\Der{R}}(\BP,\beta)$ is an epimorphism,  $\Hom_{\Der{R}}(\BP,\BP^{(I)})[i]=0$ for all $i>0$, and 
$\Hom_{\Der{R}}(\BP,R[i])=0$ for all $i\geq 2$,
it follows that $\BQ\in \BP^{\perp_{>0}}$. 

Let $Y\in \BP^{\perp_{> 0}}$. By \cite[Theorem 4.9]{AMV} we know that $\Ho^i(Y)=0$ for all $i>0$. Then 
$\Hom_{\Der{R}}(R,Y[1])=0$. Since $\Hom_{\Der{R}}(\BP^{(I)},Y[1])=0$, it follows that $\Hom_{\Der{R}}(\BQ,Y[1])=0$.

Therefore, we can apply Lemma \ref{addP} to obtain that $\BQ\in \Add\BP$.   

(2)$\Rightarrow$(1) From the existence of the triangle $R\to \BP'\to \BP''\to R[1]$, it follows that $\BP$ is a generator for $\Der{R}$. Now the conclusion follows from \cite[Theorem 4.9]{AMV}.
\end{proof}

\begin{rem}\label{envelope}
The exact sequence $R\to \Ho^0(\BP')\to \Ho^0(\BP'')\to 0$ induced in cohomology by the triangle $R\to \BP'\to \BP''\to R[1]$ is a 
$\Gen(T)$-preenvelope for $R$. Therefore, Proposition \ref{triangle} is the triangulated version of \cite[Proposition 3.11]{AMV}.
\end{rem}

\subsection{Good silting complexes}
Using the same technique as in \cite[Proposition 3.1]{Baz1} we obtain the following 

\begin{cor}\label{good-1}
Let $T$ be a silting module with respect to a morphism $\sigma$, and let $\BP$ be the silting complex associated to $\sigma$. Then there exists a silting complex $\BQ$ such that 
\begin{enumerate}[{\rm (1)}]
 \item there exists a triangle $R\to \BQ\to \BQ'\to R[1]$ such that $\BQ'$ is a direct summand of $\BQ$;
 \item the silting module  $\Ho^0(\BQ)$ generates the same torsion theory as $T$.
 \end{enumerate}
\end{cor}

\begin{proof}
(1) We start with a triangle $R\overset{\alpha}\to \BP'\overset{\beta}\to \BP''\to R[1]$. If $\BQ=\BP'\oplus \BP''^{(\omega)}$ and $\BQ'=\BP''\oplus \BP''^{(\omega)}\cong \BP''^{(\omega)}$ then we have a triangle 
$$R\overset{\alpha\oplus 0}\longrightarrow \BQ\overset{\beta\oplus 1_{\BP''^{(\omega)}}}\longrightarrow \BQ'\to R[1].$$ 
It is easy to see that $\BQ$ is partial silting, hence Lemma \ref{triangle} proves that $\BQ$ is a silting complex.

(2) Since $\Hom_{\Der{R}}(\BQ'[-1],M)=0$ for all $M\in \Gen(T)$, it follows that $\Gen(T)\subseteq \Gen(\Ho^0(\BQ))$. The converse inclusion is obvious, so we conclude that $\Gen(T)= \Gen(\Ho^0(\BQ))$.
 \end{proof}


A   torsion theory $(\T,\F)$ in $\Modr R$ is called \textsl{silting torsion theory} if there exists a silting module $S$ such that $\T=\Gen(S)$. By Corollary \ref{good-1} we know that there 
exists a silting complex $\BP$ such that the silting module $T=\Ho^0(\BP)$
generates the class $\T$ and there exists a triangle 
\[R\to \BP^n\to \BP'\to R[1]\] in $\Der R$ such that $\BP'\in\add \BP$.
Such a complex will be called a \textsl{good silting complex}.

\begin{expl}\label{compact-is-good}
 Every compact siling object is good. Indeed, if $\BP$ is silting compact, then 
 we can suppose that $P^{-1}$ and $P^0$ are finitely generated. Therefore, it is not hard to see that in the proof of Proposition \ref{triangle} 
we can find a finite set $I$ and a triangle 
$$ R\to \BQ\overset{\beta}\to \BP^{(I)}\to R[1]$$ such that $\BQ\in\Add{\BP}$. Since the class of compact objects is closed under extensions, 
it follows that $\BQ$ is compact, so $\BQ\in\add{\BP}$, hence $\BP$ is a good, cf. also \cite[Corollary 3.3]{Buan}.
\end{expl}



\subsection{Derived functors induced by silting complexes}\label{functors-der-dg}

By \cite[Example 2.1 a)]{keller-functori} we observe that the ordinary ring $R$ can be viewed as a dg-algebra concentrated in degree $0$. Therefore, a dg-module over $R$ is a 
complex of ordinary (right) $R$-modules, hence $\Mod{R}$ is the category of all complexes of $R$-modules.  
We can identify 
$\Der{R}=\Der{R}$, and we view $\BP$ as an $R$-dg-module. The complex $\BP$ is cofibrant because it is a bounded complex with projective entries. Therefore $\RHom_R(\BP,-)=\Hom_R^\bullet(\BP,-)$.  

By \cite[Example 2.1.b)]{keller-functori} $\BP$ induces a dg-algebra \[\DgEnd_R(\BP)=\Hom_R^\bullet(\BP,\BP),\] called the {\em endomorphism dg-algebra of $\BP$}. 

Let us observe that $\RHom_R(R,\BP)$ is the complex $$\dots \to 0\to \Hom_R(R,P^{-1})\overset{\sigma\circ-}\longrightarrow \Hom_R(R,{P^0})\to 0\cdots$$ which is concentrated 
in the degrees $-1$ and $0$. This complex has a canonical structure as a dg-module over the dg-algebra $\DgEnd_R(\BP)$.  Therefore $\BP$ becomes a dg-$\DgEnd_R(\BP)$-$R$-bimodule and consequently 
it induces 
 the right derived covariant functors
\[\RHom_R(\BP,-):\Der{R}\leftrightarrows\Der{\DgEnd_R(\BP)}:-\dotimes_{\DgEnd_R(\BP)}\BP.\]  

Further observe that $\DgEnd_R(\BP)=\Hom^\bullet_R(\BP,\BP)$ can be represented as the  complex
\begin{align*}0&\to \Hom_R(P^0,P^{-1})\stackrel{\left(\begin{array}{c}-\circ\sigma\\ \sigma\circ-\end{array}\right)}\la\Hom_R(P^{-1},P^{-1})\times \Hom_R(P^0,P^0)\to \\ 
&\stackrel{\left(\begin{array}{cc} \sigma\circ- & -\circ(-\sigma)\end{array}\right)}\la\Hom_R(P^{-1},P^0)\to 0
\end{align*}
which is concentrated in degrees $-1$, $0$ and $1$. From (S1) above it follows that  $\Hom_{\Der R}(\BP,\BP[1])=0$, hence  $\Ho^i(\DgEnd_R(\BP))=0$ for all $i>0$, so $\DgEnd_R(\BP)$ is homologically non-positive. We denote by $B$ the "smart" truncation of $\DgEnd_R(\BP)$, that is \[B=\bigoplus_{i\in\Z}B^i,\hbox{ where }B^i=\begin{cases}\DgEnd_R(\BP)^i\hbox{ if }i<0\\ \Zy^0(\DgEnd_R(B))\hbox{ if }i=0\\ 0\hbox{ if }i>0\end{cases}.\] Then 
$B$ is a non-positive dg-algebra and the obvious dg-algebra homomorphism $B\to\DgEnd_R(\BP)$ is actually a quasi-isomorphism. 
Hence every dg-$\DgEnd_R(\BP)$-module becomes a dg-$B$-module by restriction of scalars. As in \cite[Section 12.4]{Yek} we do not distinguish notationally between such a dg module seen as $\DgEnd_R(\BP)$-module or a $B$-module.  Moreover restriction of scalars functor is an equivalence with the quasi-inverse the induction functor, that is the derived tensor product
$-\dotimes_{B}\BP$.
Composing this equivalence with the previous adjoint pair and using the asociativity, up to a natural equivalence, of the derived tensor product, we get an adjoint pair: 
\[\RHom_R(\BP,-):\Der{R}\leftrightarrows\Der{B}:-\dotimes_B\BP.\] 

\begin{lem}\label{b-prop-rhom} The following statements are true:
\begin{enumerate}[{\rm (1)}]
 \item the functors $\RHom_R(\BP,-):\Der{R}\leftrightarrows\Der{B}:-\dotimes_B\BP$ are triangle functors;
 \item $-\dotimes_B\BP$ is a left adjoint for $\RHom_R(\BP,-)$; 
 \item $B\dotimes_B\BP \cong \BP$; 
 \item $\RHom_{B\opp}(B,\BP)\cong \BP$. 
\end{enumerate}
\end{lem}

\begin{proof}
For (1) see \cite[Proposition 24.4]{dga}. For (2) and (3) see \cite[Proposition 21.4]{dga} completed by 
\cite[Example 24.5]{dga}.
\end{proof}


\section{The silting theorem}\label{silting-thm}

\subsection{The setting and some basic properties} \label{setting-sect}\label{setting-1}\label{setting-2}
We are ready to fix some objects and homomorphisms which will be used in the following. 

Let $k$ be a commutative ring, and $R$ a $k$-algebra. We will use the following fixed objects, morphisms, and torsion pairs:
\begin{itemize}
 \item $\BP= \dots\to 0\to P^{-1}\overset{\sigma}\to {P^0}\to 0\to \cdots$ is a good silting complex 
 (hence $P^{-1}$ and $P^0$ are projective right $R$-modules); 
 \item $T=\Coker(\sigma)=\Ho^0(\BP)$ is the corresponding silting module;
\item the torsion pair generated by $T$ in $\Mod{R}$ is denoted by $\tau=(\T,\F)$;
\item we will denote by $\CH(\tau)$ be \textsl{the heart} of the $t$-structure associated to $\tau=(\T,F)$, i.e. the category of all objects $X\in \Der{R}$ which lie in triangles 
$F[1]\to X\to M\to F$, where $F\in\F$ and $M\in \T$;
 \item $\BE=\End_{\bD(R)}(\BP)$ is the endomorphism ring of $\BP$ in the derived category of $\Mod{R}$;
\item we consider the torsion pair $(\CU,\CV)$ in $\Mod{\BE}$, where $$\CU=\{X\in\Mod{\BE}\mid X\otimes_\BE T=0\};$$
 \item we fix a triangle 
$$(\dagger)\ \ \ R\overset{\alpha}\to \BP^{n}\overset{\beta}\to \BP'\overset{\gamma}\to R[1]$$ such that $\BP'\in\add{\BP}$.
\item $\DgEnd_R(\BP)$ will denote the endomorphism dg-algebra associated to $\BP$.
\end{itemize}

\begin{rem}\label{beta-stea}
Applying the functor $\Hom_{\bD(R)}(-,\BP)$ on the triangle $(\dagger)$ we obtain the exact sequence of left $\BE$-modules 
$$\Hom_{\bD(R)}(\BP',\BP)\overset{\beta^*}\to \Hom_{\bD(R)}(\BP^n,\BP)\overset{\alpha^*}\to \Hom_{\bD(R)}(R,\BP)\to 0.$$ Therefore, the above exact sequence is a projective presentation for the left $\BE$-module $T\cong \Hom_{\bD(R)}(R,\BP)$. In this setting it will be useful to consider the defect functor associated to the tensor product $$\KT_{\beta^*}=\Ker(-\otimes_{\BE}\beta^*)$$ induced by $\beta^*$.
\end{rem}

\begin{rem}\label{truncation}
As in \ref{functors-der-dg}, we consider $B$ the smart truncation of $\DgEnd_R(\BP)$. Since $B$ is non-positive, we apply \cite[Proposition 2.1]{KY} to observe that  the standard t-structure $\CH(B)$ exists in $\Der{B}$ (that is the subcategory of $\Der{B}$ consisting of 
objects concentrated in degree $0$). 
The heart of this $t$-structure in $\Der{B}$ is denoted by $\CH(B)$. It is easy to see that $\Ho^0(B)=\BE$, and it follows that $\Ho^0:\CH(B)\to\Mod\BE$ is an equivalence. 
\end{rem}

\begin{rem}\label{beta-flat}
Applying the (contravariant) triangle functor $\RHom_R(-,\BP)$ to the triangle $(\dagger)$ above, we obtain a triangle in $\Der{B\opp}$:
 \[(\ddagger)\hskip5mm B'\overset{\beta^\flat}\longrightarrow B^n\longrightarrow \BP\longrightarrow B'[1]\] 
 where the entries of this triangle are identified as $B'=\RHom_R(\BP',\BP)\in\add B$, $\ B^n=\RHom_R(\BP,\BP)^n\cong\RHom_R(\BP^n,\BP), $  $\beta^\flat=\RHom_R(\beta,\BP),$ and $\RHom_R(R,\BP)\cong\BP$. 
\end{rem}



\begin{rem}\label{pstar} If we view $\BE$ as a dg-agebra concentrated in degree $0$, then there is an obvious homomorphism of dg-algebras $p:B\to\BE$. 
Using \cite[Theorem 12.4.23(1)]{Yek}, $p$ induces the extension and the restriction of scalar functors $$p^*=-\dotimes_B\BE:\Der{B}\leftrightarrows\Der\BE:p_*,$$ and $p_*$ is the right adjoint of $p^*$. 
Note that the restriction of $p^*$ to $\CH(B)$ coincides with the restriction of $H^0$ to $\CH(B)$. Therefore, the restriction of $p_*$ at $\Mod\BE$ is a quasi-inverse of the equivalence 
$\Ho^0:\CH(B)\to\Mod\BE$. 
\end{rem}


\begin{lem}\label{rhom-heart}
If $X\in\Der{R}$ is a complex concentrated in $-1$ and $0$ then $$\Ho^{0}(\RHom_R(\BP,X))=\Hom_{\Der{R}}(\BP,X).$$

Moreover, the following are equivalent 
\begin{enumerate}[{\rm (a)}]
 \item $X\in\CH(\tau)$;
 
 \item $\RHom_R(\BP,X)\in\CH(B)$. 
\end{enumerate}
\end{lem}

\begin{proof}
The complex $X\in\CH(\tau)$ is isomorphic to a complex $$\cdots \to 0\to X^{-1}\overset{\alpha}\to X^0\to 0\to \cdots$$ which is concentrated in $-1$ and $0$ such that 
$\Ker(\alpha)\in \F$ and $\Coker(\alpha)\in \T$.
Then $\RHom_R(\BP,X)$ is the complex
\begin{align*}0&\to \Hom_R(P^0,X^{-1})\stackrel{\left(\begin{array}{c}-\circ\sigma\\ \alpha\circ-\end{array}\right)}\la\Hom_R(P^{-1},X^{-1})\times \Hom_R(P^0,X^0)\to \\ 
&\stackrel{\left(\begin{array}{cc} \alpha\circ- & -\circ(-\sigma)\end{array}\right)}\la\Hom_R(P^{-1},X^0)\to 0,
\end{align*}
and the first conclusion can be obtained by a direct computation.

(a)$\Rightarrow$(b)
Let $f:P^0\to X^{-1}$ be an $R$-morphism such that $f\sigma=0$ and $\alpha f=0$. From $f\sigma=0$ it follows that there exists $g:T\to X^{-1}$ such that $f=g\pi$, where 
$\pi:P^0\to T$ is the cokernel of $\sigma$. Since $\pi$ is an epimorphism, it follows by $\alpha f=0$ that $\alpha g=0$. Then $g$ factorizes through a morphism $T\to \Ker(\alpha)$. 
But $\Ker(\alpha)\in \F$, and we obtain $g=0$, hence $f=0$. 

Using similar techniques it follows that  $\left(\begin{array}{cc} \alpha\circ- & -\circ(-\sigma)\end{array}\right)$ is surjective. Then $\RHom_R(\BP,X)$ is in fact 
isomorphic to the complex concentrated in $0$ which is represented by $$\Ho^0(\RHom_R(\BP,X))=\Hom_{\Der{R}}(\BP,X).$$     

(b)$\Rightarrow$(a) Let $u:\Ker(\alpha)\to X^{-1}$ be the inclusion map. Suppose that $\Ker(\alpha)\notin \F$. Then there is a nonzero morphism $g:T\to \Ker(\alpha)$. 
It is easy to see that $u g\pi:P^{0}\to X^{-1}$ is a nonzero morphism which belongs to the kernel of $\left(\begin{array}{c}-\circ\sigma\\ \alpha\circ-\end{array}\right)$, a contradiction. 
Therefore, $\Ker(\alpha)\in\F$.

Let $K=\Coker(\alpha)$, and denote by $p:X^0\to K$ the canonical surjection. We will prove that $K\in \D_\sigma$. If $f:P^{-1}\to K$ is a morphism, it can be lifted to a morphism 
$g:P^{-1}\to X^0$ such that $f=pg$. Since $\left(\begin{array}{cc} \alpha\circ- & -\circ(-\sigma)\end{array}\right)$ is surjective, there exist morphisms 
$h^{i}:P^{i}\to X^{i}$, $i\in\{-1,0\}$, such that $g=\alpha h^{-1}-h^0\sigma$. It follows that $f=-ph^0\sigma$, hence $K\in\D_\sigma$. Since $D_\sigma=\T$, the proof is complete. 
\end{proof}



We recall that applying the functor $\Hom_{\Der{R}}(-,\BP)$ to the triangle $(\dagger)$ we obtain the exact sequence of left $\BE$-modules 
$$\Hom_{\bD(R)}(\BP',\BP)\overset{\beta^*}\to \Hom_{\bD(R)}(\BP^n,\BP)\overset{\alpha^*}\to \Hom_{\bD(R)}(R,\BP)\to 0.$$

\begin{lem}
\label{prop-KT-dg} Let $Y$ be a an object in $\CH(B)$.  
\begin{enumerate}[{\rm (1)}]
\item The restrictions of the functors 
$$\Ho^0(Y\dotimes_B\RHom_R(-,\BP)) \textrm{ and } \Ho^0(Y)\otimes_\BE\Hom_{\Der{R}}(-,\BP)$$ to $\add{\BP}$ are naturally isomorphic. 

\item There are natural isomorphisms of $R$-modules \[\Ho^0(Y\dotimes_B\BP)\cong\ \Ho^0(Y)\otimes_\BE T\hbox{ and }\Ho^{-1}(Y\dotimes_B\BP)\cong\Ker(\Ho^0(Y)\otimes_\BE\beta^*).\]

\item For all $i\notin\{-1,0\}$ we have $\Ho^i(Y\dotimes_B\BP)=0$. 
\end{enumerate}
\end{lem}

\begin{proof} (1) 
Let $X=\dots\to 0\to X^{-1}\overset{\rho}\to X^0\to 0 \to \dots$ be a complex from $\add{\BP}$, where $X^{-1}$ and $X^0$ are projective $R$-modules. We replace the complex $\RHom_R(X,\BP)$ by its smart truncation 
$$ \dots \to 0\to \Hom_R(X^0,P^{-1})\overset{\Phi}\to Z^0(\RHom_R(X,\BP))\to 0\to \dots,$$
where \begin{align*} Z^0(\RHom_R(X,\BP))=\left\{ \right.& (\alpha^{-1},\alpha^0)\in\Hom(X^{-1},P^{-1})\times \Hom(X^0,P^0)\mid \\   & \left. \sigma\alpha^{-1}=\alpha^0\rho\right\}.\end{align*}

Since $Y\in \CH(B)$, we can replace it by $p_*\Ho^0(Y)$. The homomorphism of dg-algebras $p:B\to\BE$ from Remark \ref{pstar} induces a ring homomorphism $p:B^0\to \BE$. It follows that we suppose that $Y$ is a complex concentrated in $0$ and $Y^0$ is the restriction along the homomorphism $p:B^0\to \BE$ of the $\BE$-module $\Ho^0(Y)$. Moreover $\Coker(\Phi)=\Hom_{\Der{R}}(X,\BP)$. 

Since $X\in\add{\BP}$, it follows that $\RHom_R(X,\BP)\in\add{B}$ is cofibrant. It follows, by using the definition of $Y\otimes^\bullet_B \RHom_R(X,\BP)$ that the functors $$\Ho^0(Y\dotimes_B\RHom_R(X,\BP))\textrm{ and }\Ho^0(Y)\otimes_{B^0}\Hom_{\Der{R}}(X,\BP)$$ are naturally isomorphic. 

Using \cite[Proposition II.2]{Bourb} we observe that, in order to complete the proof, it is enough to prove that $\Ker(p)$ is contained in the annihilators of the $B^0$ modules $\Ho^0(Y)$ and $\Hom_{\Der{R}}(X,\BP)$. For $\Ho^0(Y)$ this is obvious since $\Ho^0(Y)$ is a module obtained via the restriction of scalars functor. 

Let $(\alpha^{-1},\alpha^0)\in \Ker(p)$. It follows that there exists $s:P^0\to P^{-1}$ such that $\alpha^0=\sigma s$ and $\alpha^{-1}=s\sigma$. Note $B^0$ acts on $\Hom_{\Der{R}}(X,\BP)$ via the composition of maps (of complexes): $(\alpha^{-1},\alpha^0)\widehat{(f^{-1},f^0)}=\widehat{(\alpha^{-1}f^{-1},\alpha^0f^0)}$ (here $\widehat{(f^{-1},f^0)}$ represents the homotopy class of $(f^{-1},f^0)$). It follows that $\Ker(p) \Hom_{\Der{R}}(X,\BP)=0$, and the proof is complete. 


(2) We apply the triangle functor $Y\dotimes_B-$ to the 
  triangle $(\ddagger)$ from Remark \ref{beta-flat}.  We get a triangle 
  \[ Y\dotimes_BB'\to Y\dotimes_BB^n\to Y\dotimes_B\BP\to Y\dotimes_BB'[1].  \]
Since $Y\dotimes_BB\cong Y$, it follows that  $Y\dotimes_BB'=Y\dotimes_B\RHom_R(\BP',\BP)$ and $Y\dotimes_BB^n=Y\dotimes_B\RHom_R(\BP^n,\BP)$ are elements from $\add{Y}$. It follows that we have the following exact sequence of $k$-modules 
\begin{align*} 0=\Ho^{-1}(Y\dotimes_BB^n)&\to\Ho^{-1}(Y\dotimes_B\BP)\to\Ho^0(Y\dotimes_B B')\longrightarrow\Ho^0(Y\dotimes_BB^n)\\
&\to\Ho^0(Y\dotimes_B\BP)\to\Ho^1(Y\dotimes_BB')= 0. \end{align*}

By (1) this induces the exact sequence 
\begin{align*} 0 \to\Ho^{-1}(Y\dotimes_B\BP)&\to\Ho^0(Y) \otimes_{\BE}\Hom_{\Der{R}}(\BP',\BP) \stackrel{\Ho^0(Y)\otimes_\BE\beta^*}\longrightarrow \\ & \stackrel{\Ho^0(Y)\otimes_\BE\beta^*}\longrightarrow\Ho^0(Y)\dotimes_{\BE}\Hom_{\Der{R}}(\BP',\BP)\to\Ho^0(Y\dotimes_B\BP)\to 0,\end{align*}
and the conclusion is now clear.



(3) This is a consequence of the proof of (2). 
  \end{proof}

\begin{rem}  
By the statement (2) in the above lemma it follows that the functor $\KT_{\beta^*}=\Ker(-\otimes_{\BE}\beta^*)$ defined in 
\ref{setting-1} acts actually between $\Mod{\BE}$ and $\Mod{R}$. Using a similar proof as in \cite[Proposition 4]{BrZ}, 
it follows that it plays a similar role with the role the functor $\Tor_1^\BE(-,T)$ for the case when $T$ is of flat dimension at most 1. We recall that if $T$ is a tilting module then its flat dimension as a left $\End(T)$-module is at most $1$.        
\end{rem}

\subsection{The silting theorem for derived categories}

We will denote $$\K=\Ker(-\dotimes_B \BP)\subseteq\Der{B},$$ and 
$$\K^\perp=\{Y\in \Der{B}\mid \Hom_{\Der{B}}(X,Y[n])=0 \text{ for all } X\in\K \text{ and } n\in\Z\}.$$ 

The silting theorem can be formulated in the following way:

\begin{thm}\label{silting-thm-1}
Let $\BP$ be a good silting complex as in Setting \ref{setting-1}. Then
\begin{enumerate}[{\rm (1)}]
\item  
the functor $\RHom_R(\BP,-)$ induces an equivalence
$$\RHom_R(\BP,-):\Der{R}\leftrightarrows\K^\perp,$$ 
 and  $-\dotimes_B\BP:\K^\perp\to \Der{R}$ is a quasi-inverse for $\RHom_R(\BP,-)$;
 
\item the restrictions of these functors to $\CH(\tau)$ and $\CH(B)\cap K^\perp$ induce an equivalence
\[\RHom_R(\BP,-):\CH(\tau)\leftrightarrows \CH(B)\cap \K^\perp:-\dotimes_B\BP.\]
 \end{enumerate}
 \end{thm}

\begin{proof}
(1) Let us denote by $\gamma$ and $\delta$ the unit, respectively the counit, associated to the adjunction 
$(-\dotimes_B \BP)\dashv \RHom_R(\BP,-)$. Then the map $\gamma_B:B\to\RHom_R(\BP,B\dotimes_B\BP)$ is an isomorphism, and the triangle $(\dagger)$  
implies that $R$ lies in the smallest thick subcategory containing $\BP$. Therefore the condition 4) from \cite[Theorem 6.4]{Ni-Sa} holds true. 
By the (equivalent) condition 3) of the above cited Theorem it follows that 
$\RHom_R(\BP,-)$ is fully faithful, hence
$\delta:\RHom_R(\BP,-)\dotimes_B \BP\to 1_{\Der{R}}$ is an isomorphism. From the adjunction isomorphism  
$\Hom_{\Der B}(Y,\RHom_R(\BP,X))\cong\Hom_{\Der R}(Y\dotimes_B\BP,X)$ we obtain $\RHom_R(\BP,\Der{R})\subseteq \K^\perp$.

Conversely, if 
$Y\in\K^\perp$, we have $\delta_{Y\dotimes_B\BP}(\gamma_Y\dotimes_B\BP)=1_{Y\dotimes_B\BP}$.  Since $\RHom_R(\BP,-)$ is fully faithful, it follows that $\delta_{Y\dotimes_B\BP}$ is an isomorphism. Then $\gamma_Y\dotimes_B\BP$ is an isomorphism. 
Therefore, completing $\gamma_Y$ to a triangle 
$$Z\overset{\alpha}\to Y\overset{\gamma_Y}\la \RHom_R(\BP,Y\dotimes_B\BP)\overset{\beta}\to Z[1],$$
it follows $Z\dotimes_B\BP=0$. This implies that $Z\in \K$, hence $\alpha=0$ and $\alpha[1]=0$. It follows that $\beta$ is a split homomorphism. Since $\RHom_R(\BP,Y\dotimes_B\BP)\in \K^\perp$ and $Z[1]\in K$, this is possible only if $Z=0$. Then $\gamma_Y$ is an isomorphism.

Therefore $\RHom_R(\BP,\Der{R})= \K^\perp$. This shows that the functors 
\[\RHom_R(\BP,-):\Der R\leftrightarrows\Der{B}\cap\K^\perp:-\dotimes_B\BP\] 
induce mutually inverse equivalences.

(2) Using Lemma \ref{rhom-heart} it follows that for every $X\in\CH(\tau)$ we have $\RHom_R(\BP,X)\in \CH(B)$. 

Conversely, let $Y\in \CH(B)\cap\K^\perp$. By using Lemma \ref{prop-KT-dg}, we observe that $Y\dotimes_B\BP$ is a complex concentrated in $-1$ and $0$. Since $\RHom_R(\BP,Y\dotimes_B\BP)\cong Y$, we can apply Lemma \ref{rhom-heart} one more time to conclude that $Y\dotimes_B\BP\in \CH(\tau)$, and the proof is complete.
\end{proof}

\subsection{The silting counter equivalence}

We have seen in Theorem \ref{silting-thm-1} that $\Ho^0\RHom_R(\BP,-)$ is an equivalence between $\CH(\tau)$ and an abelian subcategory of $\Mod\BE$.  
From \cite[Chapter I, Corollary 2.2]{HRS}, the  pair $(\F[1],\CT)$ is a torsion pair in the abelian category $\CH(\tau)$. It  induces a torsion pair in the abelian subcategory 
$\Ho^0\RHom_R(\BP,\CH(\tau))$ of $\Mod{\BE}$. In the following we will describe, as in the tilting case, this torsion pair by using a natural torsion pair inced by $\BP$ on $\Mod\BE$.

By applying the functor $p^*$ from Remark 
\ref{pstar}, we get a subcategory $p^*(\K)$ of $\Der\BE$. As before, we will use the notation 
$$p^*(\K)^\perp=\{Y\in \Der{\BE}\mid \Hom_{\Der{\BE}}(X,Y[n])=0 \text{ for all } X\in p^*(\K) \text{ and } n\in\Z\}.$$ 

\begin{lem}\label{pK-perp}
Using the above notations we have 
$$\Ho^0(\CH(B)\cap \K^\perp)=\Mod\BE\cap p^*(\K)^\perp.$$
\end{lem}

\begin{proof}
Let $Y\in\CH(B)\cap\K^\perp$. We use the adjunction
$p^*\dashv p_*$ and the fact that $\Ho^0$ is an equivalence with the inverse $p_*$, in order to obtain:
\[\Hom_{\Der\BE}(p^*(\K),\Ho^0(Y))\cong\Hom_{\Der{B}}(\K,p_*(\Ho^0(Y))\cong\Hom_{\Der{B}}(\K,Y)=0,\] so $\Ho^0(Y)\in p^*(\K)^\perp$. 

Conversely, for $Z\in\Mod\BE\cap p^*(\K)^\perp$, we denote $Y=p_*(Z)\in\CH(B)$ and we have $\Ho^0(Y)\cong Z$. Then 
\[\Hom_{\Der{B}}(\K,Y)\cong\Hom_{\Der{B}}(\K,p_*(Z))\cong\Hom_{\Der\BE}(p^*(\K),Z)=0,\] and it follows that $Y\in\K^\perp$. 
\end{proof}

\begin{thm}\label{1-silting-equiv} The following statements are true.
\begin{enumerate}[{\rm (1)}] 
 \item 
The functor
$$\Hom_{\Der{R}}(\BP,-):\CH(\tau)\to \Mod{\BE}\cap p^*(\K)^\perp$$
induces an equivalence of categories, whose quasi-inverse is 
$(-\dotimes_B\BP)\circ p_*$.

\item The restrictions of the above functors induce the equivalences 
\begin{enumerate}[{\rm (a)}]
\item $\Def_{\sigma}(-)=\Hom_{\Der{R}}(\BP,-[1]):\F\leftrightarrows \CU\cap p^*(\K)^\perp: \KT_{\beta^*}(-),$ and

\item $\Hom_R(T,-)=\Hom_{\Der{R}}(\BP,-):\T\leftrightarrows \CV\cap p^*(\K)^\perp: -\otimes_\BE T.$
\end{enumerate}

\end{enumerate}
\end{thm}

\begin{proof} 
(1) By Lemma \ref{rhom-heart} it follows that the restrictions of the functors $\Ho^0(\RHom_R(\BP,-))$ and $\Hom_{\Der R}(\BP,-)$ to $\CH(\tau)$ coincide. Now the conclusion follows from Theorem \ref{silting-thm-1} and Lemma \ref{pK-perp} since we have 
\begin{align*} \Hom_{\Der R}(\BP,\CH(\tau))& =\Ho^0(\RHom_R(\BP,\CH(\tau)))=\Ho^0(\CH(B)\cap \K^\perp)\\ &=\Mod\BE\cap p^*(\K)^\perp 
.\end{align*}

%

(2) Note that if $X\in\CH(\tau)$ then the exact sequence associated to $X$ which is induced by the torsion pair $(\F[1],\T)$ is $$0\to\Ho^{-1}(X)[1]\to X\to\Ho^0(X)\to0.$$ Hence 
$X\in\F[1]$ (respectively $X\in\T$) \iff $\Ho^0(X)=0$ ($\Ho^{-1}(X)=0$).

(2)(a) Fix an object $X\in\CH(\tau)$. The natural map $$\delta_X:\RHom_R(\BP,X)\dotimes_B\BP\to X$$ is an isomorphism. By Lemma \ref{prop-KT-dg}, we have the isomorphisms 
\begin{align*} \Ho^0(X)&\cong\Ho^0(\RHom_R(\BP,X)\dotimes_B\BP)\cong \Ho^0(\RHom_R(\BP,X))\otimes_\BE T\\&= \Hom_{\Der R}(\BP,X)\otimes_\BE T.\end{align*}
As we have seen, $X\in\F[1]$ \iff $\Ho^0(X)=0$, which is further equivalent 
$\Hom_{\Der R}(\BP,X)\in\Ker(-\otimes_\BE T)=\U$.  Therefore the equivalence from (1) induces the equivalence 
\[\Hom_{\Der R}(\BP,-[1]):\F\to\U\cap p^*(\K)^\perp,\] whose quasi-inverse is  $p_*(-[-1])\dotimes_B\BP)$. 

Moreover, for every $Z\in \U\cap p^*(\K)^\perp$, since $p_*(Z)\dotimes_B\BP$ is concentrated in $-1$, we obtain from Lemma \ref{prop-KT-dg}(2) the natural isomorphisms
$$\Ho^{-1}(p_*(Z)\dotimes_B\BP)\cong \Ker(\Ho^0p_*\otimes_{\BE}\beta^*)\cong\Ker(Z\otimes_\BE\beta^*).$$
Therefore, the restrictions of functors $p_*(-)\dotimes_B\BP$ and $\KT_{\beta^*}$ to $\U\cap p^*(\K)^\perp$ are natural isomorphic, and the proof is complete.   

(2)(b) Let $E$ be the endomorphism ring of $T$, and $\varphi:\BE\to E$ be the canonical surjective ring homomorphism. 

If $M\in \T$ then the right $\BE$-module $\Hom_{\Der{R}}(\BP,M)$ is the module induced by the restriction of scalars along $\varphi$ of the $E$-module $\Hom_R(T,M)$. Moreover, if $X\in\Mod{\BE}$ is a module such that $X\otimes_\BE T=0$, then the induced $E$-module $X'=X\otimes_\BE E$ has the property $X'\otimes_E T=0$. By \cite[Proposition 3.2]{AMV} we conclude that $T$ is tilting as an $R/\textrm{Ann}(T)$-module. It follows from the tilting theorem proved in \cite[Theorem 4.5]{Baz1} that $\Hom_E(X',\Hom_R(T,M))=0$. Using the canonical adjunction isomorphisms, we obtain the equality $\Hom_\BE(X, \Hom_{\Der{R}}(\BP,M))=0$, hence $\Hom_{\Der{R}}(\BP,M)\in\CV\cap p^*(K)^\perp$.

Let $X\in \CV\cap p^*(K)^\perp$. Then there exists an object $L\in \CH(\tau)$ such that $\Hom_{\Der{R}}(\BP,L)=X$. Since $(\F[1],\T)$ is a torsion pair in $\CH(\tau)$, there exists a short exact sequence in $\CH(\tau)$ of the form 
$0\to F[1]\to L\to M\to 0$, where $F\in \F$ and $M\in \T$. We apply the functor $\Hom_{\Der{R}}(\BP,-):\CH(\tau)\to \Mod{\BE}$ to this exact sequence. Since $\Hom_{\Der{R}}(\BP,-)$ is an equivalence of categories from $\CH(\tau)$ to a full subcategory of $\Mod{\BE}$,  
we obtain the short exact sequence 
$$0\to \Hom_{\Der{R}}(\BP,F[1])\to \Hom_{\Der{R}}(\BP,L)\to \Hom_{\Der{R}}(\BP,M)\to 0$$ in $\Mod{\BE}$. But 
$\Hom_{\Der{R}}(\BP,F[1])\in \U$ and $\Hom_{\Der{R}}(\BP,L)\cong X\in \CV$. This implies that 
$\Hom_{\Der{R}}(\BP,L)\to \Hom_{\Der{R}}(\BP,M)$ is an isomorphism, hence $L\cong M$ belongs to $\T$. 

It follows that $\Hom_{\Der{R}}(\BP,\T)=\CV\cap p^*(K)^\perp$. Applying Lemma \ref{prop-KT-dg} it is easy to see that 
for every $X\in \CV\cap p^*(K)^\perp$ the complex $p_*(X)\dotimes_B\BP$ is concentrated in $0$. Moreover, we have $\Ho^0(p_*(X)\dotimes_B\BP)=X\otimes_\BE T$. Therefore, the functor $-\otimes_\BE T$ is a quasi-inverse of the functor $\Hom_R(T,-)=\Hom_{\Der{R}}(\BP,-):\T\to \CV\cap p^*(K)^\perp$.
\end{proof}

\begin{cor}
If $\BP$ is a compact silting complex then  we have the equivalences:
\begin{enumerate}[{\rm (a)}] 
 \item 
$\Hom_{\Der{R}}(\BP,-):\CH(\tau)\leftrightarrows \Mod{\BE}:-\dotimes \BP$,
\item $\Hom_R(T,-)=\Hom_{\Der{R}}(\BP,-):\T\leftrightarrows \CV: -\otimes_\BE T=-\otimes_E T,$
and 

\item $\Def_{\sigma}(-)=\Hom_{\Der{R}}(\BP,-[1]):\F\leftrightarrows \CU: -[-1]\dotimes_B \BP=\KT_{\beta^*}(-),$
where $\KT_{\beta^*}$ is computed with respect a (fixed) triangle of the form $(\dagger)$.
\end{enumerate}
\end{cor}

\begin{proof} 
As we have seen in Example \ref{compact-is-good}, the compact silting complex $\BP$ is good, hence we can use Theorems \ref{silting-thm-1} and \ref{1-silting-equiv}.

We apply $\RHom_{R}(-,\BP)$ to the triangle 
$P^{-1}\to P^0\to \BP\to P^{-1}[1]$, and we obtain a triangle of left $B$-modules $X\to Y\to B\to X[1]$, with $X,Y\in \add{\BP}$. If $Z\in \K$ then 
$X\dotimes_B Z=Y\dotimes_B Z=0$. It follows that $Z\cong B\dotimes_B Z=0$, and the proof is complete.  
\end{proof}

\begin{rem} The compact case was discovered in \cite[Theorem 2.15]{HKM}, where the authors proved  directly that $\Hom_{\Der{R}}(\BP,-):\CH(\tau)\rightarrow \Mod{\BE}$ 
is fully faithful. Our approach has the advantage that we are able to compute the quasi-inverse of $\Hom_{\Der{R}}(\BP,-)$.
\end{rem}

\end{document}